\documentclass[11pt]{article}
\usepackage{verbatim} 
\usepackage{amsmath}
\usepackage{amssymb}
\usepackage{amsfonts}
\usepackage{amsthm}
\usepackage{cite}
\usepackage{graphicx} 
\usepackage{pgf,tikz}
\usepackage{tikz-cd} 
\usepackage{mathrsfs}
\usetikzlibrary{arrows}
\usepackage{color}
\usepackage[legalpaper, margin=1in]{geometry}
\usepackage{enumitem}
\usepackage{latexsym}

\newtheorem{theorem}{Theorem}[section]
\newtheorem*{theorem*}{Theorem}
\newtheorem{thmx}{Theorem}

\newtheorem{corollary}[theorem]{Corollary}
\newtheorem{lemma}[theorem]{Lemma}
\newtheorem{proposition}[theorem]{Proposition}

\theoremstyle{definition}

\newtheorem{definition}[theorem]{Definition}

\newtheorem{question}[theorem]{Question}
\newtheorem{claim}{Claim}[theorem]
\theoremstyle{remark}
\newtheorem*{claimproof}{Proof}

\newcommand{\R}{\mathbb{R}}
\newcommand{\N}{\mathbb{N}}

\newcommand{\claimend}{{\hfill $\blacksquare$}}

\renewcommand{\restriction}{\mathord{\upharpoonright}}

\title{Borel Measurable Hahn-Mazurkiewicz Theorem} 
\author{Jan Dud\'ak\footnote{https://orcid.org/0000-0003-0627-6641}
\footnote{Supported by the grant SVV-2020-260583}\\
Department of Mathematical Analysis\\ 
Faculty of Mathematics and Physics, Charles University\\
Prague, Czechia\\
E-mail: dudakjan@seznam.cz
\and 
Benjamin Vejnar\footnote{https://orcid.org/0000-0002-2833-5385} \footnote{
Supported by the grant SVV-2020-260583
}\\
Department of Mathematical Analysis\\ 
Faculty of Mathematics and Physics, Charles University\\
Prague, Czechia\\
E-mail: vejnar@karlin.mff.cuni.cz}

\begin{document}
\maketitle

\begin{abstract}
It is well known due to Hahn and Mazurkiewicz that every Peano continuum is a continuous image of the unit interval.
We prove that an assignment, which takes as an input a Peano continuum and produces as an output a continuous mapping whose range is the Peano continuum, can be realized in a Borel measurable way. Similarly, we find a Borel measurable assignment which takes any nonempty compact metric space and assigns a continuous mapping from the Cantor set onto that space. To this end we use the Burgess selection theorem. Finally, a Borel measurable way of assigning an arc joining two selected points in a Peano continuum is found.
\end{abstract}


\section{Introduction}

A lot of results in mathematics are of the form $\forall a\in A \, \exists b\in B: T(a, b)$. In many cases the sets $A$ and $B$ can be equipped with natural topologies or standard Borel structures. Then it makes sense to ask whether there is a continuous or Borel measurable mapping $b:A \to B$ satisfying $T(a, b(a))$ for every $a\in A$. A natural way to prove this kind of result is to use a suitable selection (resp. uniformization) theorem applied to the set $\{(a, b); \, T(a,b)\}$. Some of the most useful selection theorems for this matter include results by Kuratowski and Ryll-Nardzewski, Kunugui and Novikov, or Arsenin and Kunugui \cite[theorems 12.13, 28.7, 35.46]{Kechris}. However, it may happen that none of the above selection theorems can be directly applicable.

Dugunji extension theorem provides a concrete example of the idea expressed in the first paragraph: fixing a metric space $Y$ and its closed subspace $X$, it is well known that every bounded continuous mapping $f:X\to\mathbb R$ can be extended to a bounded continuous mapping $F: Y\to\mathbb R$ by the Tietze extension theorem. Dugundji proved that this assignment $f\in C^*(X)\mapsto F\in C^*(Y)$ can be realized in a continuous way with respect to the uniform topology (hence it is Borel measurable) and linear at the same time (see \cite{Dugundji} or \cite[Theorem 6.6.2 and Remark 6.6.3]{vanMill}).

The main results of this paper provide Borel measurable ways to the following classical results:
(a) every nonempty compact metrizable space is a continuous image of the Cantor set,
(b) every Peano continuum is a continuous image of $[0,1]$,
(c) every two distinct points in a Peano continuum are end-points of an arc.

\begin{thmx}[with details in Theorem \ref{CantorImagesThm}]
There is a Borel measurable way of assigning to every nonempty compact metrizable space $K$ a continuous surjective mapping $f: \mathcal{C} \to K$, where $\mathcal{C}$ is the Cantor space.
\end{thmx}

\begin{thmx}[with details in Theorem \ref{MainTheorem}]
There is a Borel measurable way of assigning to every Peano continuum $K$ a surjective continuous mapping $f: [0,1]\to K$.
\end{thmx}

\begin{thmx}[with details in Theorem \ref{ArcsInPeanoContinua}]
There is a Borel measurable way of assigning to every Peano continuum $K$ and a pair of distinct points $x,y\in K$ an arc $A$ in $K$ with end-points $x$ and $y$. 
\end{thmx}

Theorem A is obtained by an application of a selection theorem by Burgess \cite{Burgess}. Surprisingly, we were not able to apply any selection principle to prove Theorem B. Thus we followed the proof of the Hahn-Mazurkiewicz theorem as written in \cite{Nadler} and we verified that all the steps can be carried out in a Borel measurable way. Theorem C is a nontrivial consequence of Theorem B.


Theorems A and B imply some consequences in the context of invariant descriptive set theory (see \cite{Gao}). Namely it follows that the space of all continuous mappings from $[0,1]$ (or the Cantor set) into the Hilbert cube with the topology of uniform convergence provides a new yet equivalent coding for the collection of all Peano continua (or compact metrizable spaces) when considered naturally as a subspace of the hyperspace of the Hilbert cube with the Vietoris topology. Details are included in Corollary \ref{CorMainTheorem}. This gives a parallel result to that of Gao for separable complete metric spaces which can be represented either as closed subspaces of the Urysohn space with the Effros Borel structure or as metrics on $\N$ with the topology of pointwise convergence \cite[Theorem 14.1.3]{Gao}.


It was proved independently by Moise \cite{Moise} and Bing \cite{Bing}, as an answer to a question by Menger \cite{Menger}, that every Peano continuum admits a convex metric. A natural question related to the main focus of this paper follows.

\begin{question}
Is it possible to assign a convex metric to every Peano continuum in a Borel way?
\end{question}





\section{Preliminaries}

In this section we introduce the notation, terminology, definitions and basic facts which will be used throughout this paper. By a natural number we mean a strictly positive integer. We denote the set of natural numbers by $\N$. In addition, we use the symbol ${\N}_0$ to denote the set of non-negative integers, i.e. ${\N}_0 = \N \cup \{ 0 \}$. Also, we denote by ${\R}^+$ the set of strictly positive real numbers.

A subset of a topological space $X$ is said to be Borel if it belongs to the smallest $\sigma$-algebra on $X$ containing every open subset of $X$. For any two topological spaces $X$ and $Y$, a mapping $f \colon X \to Y$ is said to be Borel measurable if $f^{-1} (U)$ is Borel in $X$ for every open subset $U$ of $Y$. A Polish space is a separable completely metrizable topological space. It is a well-known fact that a subspace $Y$ of a Polish space $X$ is Polish if and only if $Y$ is $G_{\delta}$ in $X$. A Polish group is a topological group which is also a Polish space. A continuum is a compact connected metrizable topological space. We do not consider the empty topological space to be connected. Therefore, in particular, every continuum is a nonempty space. A Peano continuum is a locally connected continuum.

We denote by $\mathcal{C}$ the Cantor space, i.e. the space $\{ 0,1 \}^{\N}$ equipped with the product topology. Recall that a topological space $X$ is homeomorphic to $\mathcal{C}$ if and only if $X$ is a nonempty, compact, metrizable, zero-dimensional space with no isolated points (this is a classical theorem due to Brouwer). We denote by $I$ the compact interval $[0,1]$. The Hilbert cube, denoted by $Q$, is the space $I^{\N}$ endowed with the product topology.

For a Polish space $X$, we denote by $\mathcal{K} (X)$ the space of all nonempty compact subsets of $X$ equipped with the Vietoris topology (and the Hausdorff metric). It is well-known that $\mathcal{K} (X)$ is a Polish space. Moreover, $\mathcal{K} (X)$ is compact if and only if $X$ is compact. Let us consider the following two subspaces of $\mathcal{K} (X)$: We denote by $\mathsf{C} (X)$ the space of all continua in $X$ and by $\mathsf{LC} (X)$ the space of all Peano continua in $X$. It is easy to see that $\mathsf{C} (X)$ is a closed subset of $\mathcal{K}(X)$. In particular, it is Borel. The set $\mathsf{LC} (X)$ is Borel in $\mathcal{K} (X)$ as well (see \cite{GvM} for reference).

Recall that every separable metrizable space is homeomorphic to a subspace of $Q$. In particular, $Q$ contains a homeomorphic copy of every metrizable compact space. So, in this sense, the space $\mathcal{K}(Q)$ represents the class of nonempty metrizable compact spaces. Similarly, the classes of continua and Peano continua are represented by $\mathsf{C} (Q)$ and $\mathsf{LC} (Q)$ respectively.

For any metrizable compact space $X$ and any Polish space $Y$, we denote by $C(X,Y)$ the set of all continuous mappings from $X$ to $Y$ and we equip $C(X,Y)$ with the topology of uniform convergence (equivalently, the compact-open topology). It is well-known that $C(X,Y)$ is a Polish space. We denote by $\mathcal{E}(X,Y)$ the subspace of $C(X,Y)$ consisting of injective mappings. Note that by the compactness of $X$, we have
\[ \mathcal{E}(X,Y)=\big\lbrace f \in C(X,Y) \, ; \ f \textup{ is a homeomorphic embedding of } X \textup{ into } Y \big\rbrace. \]
It is not difficult to show that $\mathcal{E}(X,Y)$ is $G_{\delta}$ in $C(X,Y)$. Hence, $\mathcal{E}(X,Y)$ is a Polish space.

For any set $X$ and any equivalence relation $E$ on $X$, a subset $M$ of $X$ is said to be $E$-invariant if $[x]_E \subseteq M$ for every $x \in M$. A transversal for $E$ is a subset of $X$ containing exactly one element from each $E$-equivalence class. 

Recall that for any group $G$ and any set $X$, a mapping $\alpha \colon G \times X \to X$ is said to be an action of $G$ on $X$ if $\alpha (e,x)=x$ and $\alpha \big( g, \alpha (h,x) \big) = \alpha(gh,x)$ for all $g,h \in G$ and $x \in X$, where $e$ is the identity element of $G$. If $\alpha$ is an action of $G$ on $X$, the equivalence relation $E$ on $X$ defined by
\[ xEy \iff \exists \, g \in G : \alpha (g,x)=y \]
is called the orbit equivalence relation induced by $\alpha$.

An equivalence relation $E$ on a Borel subset $Y$ of a Polish space $X$ is said to be countably separated if there is a sequence $(Z_n)_{n=1}^{\infty}$ of $E$-invariant Borel subsets of $Y$ such that for all $x,y \in Y$, the points $x$ and $y$ are $E$-equivalent if and only if $\{ n \in \N \, ; \ x \in Z_n \} = \{ n \in \N \, ; \ y \in Z_n \}$.

In the remaining part of this chapter we present various lemmata related to spaces of compact sets and spaces of continuous mappings.

\begin{lemma}\label{CharacterizationOfBorelnessIntoVietoris}
Let $X$ be a topological space and $Y$ a Polish space. For any mapping $f \colon X \to \mathcal{K}(Y)$, the following four assertions are equivalent:
\begin{enumerate}[label=(\roman*),font=\textup, noitemsep]
    \item $f$ is Borel measurable;
    \item the set $\big\lbrace x \in X ; \, f(x) \subseteq V \big\rbrace$ is Borel for every open set $V \subseteq Y$;
    \item the set $\big\lbrace x \in X ; \, f(x) \cap V \neq \emptyset \big\rbrace$ is Borel for every open set $V \subseteq Y$;
    \item the set $\big\lbrace x \in X ; \, f(x) \cap F \neq \emptyset \big\rbrace$ is Borel for every closed set $F \subseteq Y$.
\end{enumerate}
\end{lemma}
\begin{proof}
Let $\mathcal{A}_V := \big\lbrace K \in \mathcal{K}(Y) \, ; \ K \cap V \neq \emptyset \big\rbrace$ and $\mathcal{B}_V := \big\lbrace K \in \mathcal{K}(Y) \, ; \ K \subseteq V \big\rbrace$ for every open set $V \subseteq Y$. Define $\mathscr{S}_1 := \{ \mathcal{A}_V \, ; \ V \subseteq Y \textup{ is open} \}$ and $\mathscr{S}_2 := \{ \mathcal{B}_V \, ; \ V \subseteq Y \textup{ is open} \}$. By the definition of the Vietoris topology, the family $\mathscr{S}_1 \cup \mathscr{S}_2$ is a subbase for $\mathcal{K}(Y)$. Thus, since $\mathcal{K}(Y)$ is separable and metrizable, (i) is equivalent to the conjunction of (ii) and (iii). This is a consequence of the general fact that a mapping from a topological space to a separable metrizable space is Borel measurable iff the preimage of every subbasic set is Borel. Moreover, it is clear that (ii) is equivalent to (iv).

It remains to show that (ii) is equivalent to (iii). Assume that (ii) holds and let $V \subseteq Y$ be open. We are going to prove that $\big\lbrace x \in X ; \, f(x) \cap V \neq \emptyset \big\rbrace$ is a Borel set. Since $Y$ is metrizable, there are closed sets $F_1, F_2, \dotsc \subseteq Y$ such that $V=\bigcup_{n \in \N} F_n$. Denoting $V_n := Y \setminus F_n$, $n \in \N$, we have
\begin{gather*}
    \big\lbrace x \in X ; \, f(x) \cap V \neq \emptyset \big\rbrace = \bigcup_{n \in \N} \big\lbrace x \in X ; \, f(x) \cap F_n \neq \emptyset \big\rbrace\\
    = \bigcup_{n \in \N} \Big( X \setminus \big\lbrace x \in X ; \, f(x) \subseteq V_n \big\rbrace \Big) = X \setminus \bigcap_{n \in \N} \big\lbrace x \in X ; \, f(x) \subseteq V_n \big\rbrace.
\end{gather*}
Since (ii) holds and each of the sets $V_1,V_2,\dotsc$ is open, we are done.

The implication from (iii) to (ii) can be proven in a similar fashion.
\end{proof}

\begin{lemma}\label{closureOfUnionIsBorel}
Let $X$ be a topological space and $Y$ a compact metrizable space. Let $f_n \colon X \to \mathcal{K}(Y)$, $n \in \N$, be Borel measurable mappings. Then the mapping $f \colon X \to \mathcal{K}(Y)$ defined by
\[ f(x)=\overline{\bigcup_{n \in \N} f_n (x)} \]
is Borel measurable.
\end{lemma}
\begin{proof}
Clearly, for every open set $V \subseteq Y$, we have
\[ \big\lbrace x \in X ; \, f(x) \cap V \neq \emptyset \big\rbrace = \bigg\lbrace x \in X ; \ V \cap \bigcup_{n \in \N} f_n (x) \neq \emptyset \bigg\rbrace = \bigcup_{n \in \N} \big\lbrace x \in X ; \ f_n(x) \cap V \neq \emptyset \big\rbrace. \]
An application of Lemma \ref{CharacterizationOfBorelnessIntoVietoris} finishes the proof.
\end{proof}

\begin{lemma}\label{intersectionIsBorel}
For any Polish space $X$, the mapping from
\[ \big\lbrace (K_1,K_2) \in \mathcal{K}(X) \times \mathcal{K}(X) \, ; \ K_1 \cap K_2 \neq \emptyset \big\rbrace \]
to $\mathcal{K}(X)$ given by $(K_1,K_2) \mapsto K_1 \cap K_2$ is Borel measurable.
\end{lemma}
\begin{proof}
Denote by $\mathcal{A}$ the domain of the mapping in question and let $V\subseteq X$ be an open set. We are going to show that the set $\mathcal{G} := \big\lbrace (K_1,K_2) \in \mathcal{A} \, ; \ K_1 \cap K_2 \subseteq V \big\rbrace$ is relatively open (and hence relatively Borel) in $\mathcal{A}$. Let $(K_1,K_2) \in \mathcal{G}$ be given. Then $K_1 \cap K_2 \subseteq V$, which implies that $K_1 \setminus V$ and $K_2 \setminus V$ are disjoint closed sets. Therefore, there exist disjoint open sets $V_1,V_2 \subseteq X$ such that $K_1 \setminus V \subseteq V_1$ and $K_2 \setminus V \subseteq V_2$. Let $\mathcal{U} := \big\lbrace (L_1,L_2) \in \mathcal{A} \, ; \ L_1 \subseteq V \cup V_1 \textup{ and } L_2 \subseteq V \cup V_2 \big\rbrace$. Clearly, $\mathcal{U}$ is relatively open in $\mathcal{A}$ and it contains $(K_1,K_2)$. Moreover, it is easy to see that $\mathcal{U} \subseteq \mathcal{G}$.

By Lemma \ref{CharacterizationOfBorelnessIntoVietoris}, we are done.
\end{proof}

\begin{lemma}\label{IntersectionOfBorelMapsToHyperspaceIsBorel}
Let $X$ be a topological space and let $Y$ be a Polish space. Assume that $f_n \colon X \to \mathcal{K}(Y)$, $n \in \N$, are Borel measurable mappings such that $\bigcap \big\lbrace f_n(x) ; \, n \in \N \big\rbrace \neq \emptyset$ for every $x \in X$. Then the mapping $f \colon X \to \mathcal{K}(Y)$ given by $f(x) = \bigcap \big\lbrace f_n(x) ; \, n \in \N \big\rbrace$ is Borel measurable.
\end{lemma}
\begin{proof}
For every $n \in \N$, define a mapping $g_n \colon X \to \mathcal{K}(Y)$ by $g_n (x) = f_1 (x) \cap \dots \cap f_n (x)$. Then $g_1 = f_1$ and $g_{n+1} (x) = g_n (x) \cap f_{n+1} (x)$ for every $n \in \N$ and $x \in X$. Thus, it easily follows from Lemma \ref{intersectionIsBorel} that each of the mappings $g_1, g_2, g_3, \dotsc$ is Borel measurable. Moreover, $f$ can be shown to be the pointwise limit of the sequence $(g_n)_{n=1}^{\infty}$. Therefore, $f$ is Borel measurable.
\end{proof}



\begin{lemma}\label{CtsFunctionAndCompactSetMapstoImageIsCts}
Let $X$ be a compact metrizable space and $Y$ a Polish space. Then the mapping from $C(X,Y) \times \mathcal{K}(X)$ to $\mathcal{K}(Y)$ given by $(f,K) \mapsto f(K)$ is continuous.
\end{lemma}
\begin{proof}
By the definition of the Vietoris topology, it suffices to show that for every open set $V \subseteq Y$, the sets
\begin{align*}
    \mathcal{A}_V &:= \big\lbrace (f,K) \in C(X,Y) \times \mathcal{K}(X) \, ; \ f(K) \subseteq V \big\rbrace,\\
    \mathcal{B}_V &:= \big\lbrace (f,K) \in C(X,Y) \times \mathcal{K}(X) \, ; \ f(K) \cap V \neq \emptyset \big\rbrace
\end{align*}
are open in $C(X,Y) \times \mathcal{K}(X)$. Let $V$ be given. To show that $\mathcal{A}_V$ is open, let $(f,K) \in \mathcal{A}_V$. Then $f(K)$ is a closed set contained in $V$. Thus, there is an open set $U \subseteq Y$ with $f(K) \subseteq U \subseteq \overline{U} \subseteq V$. Since $f$ is continuous, there is an open set $G \subseteq X$ such that $K \subseteq G$ and $f(G) \subseteq U$. Let
\[ \mathcal{U}:=\big\lbrace g \in C(X,Y) \, ; \ g\big(\overline{G}\big) \subseteq V \big\rbrace \times \big\lbrace L \in \mathcal{K}(X) \, ; \ L \subseteq G \big\rbrace.\]
Then $\mathcal{U}$ is open in $C(X,Y) \times \mathcal{K}(X)$ and $\mathcal{U} \subseteq  \mathcal{A}_V$. Moreover, since $f\big(\overline{G}\big) \subseteq \overline{f(G)}\subseteq \overline{U} \subseteq V$ and $K \subseteq G$, we have $(f,K) \in \mathcal{U}$.

Now, let us show that $\mathcal{B}_V$ is open. To that end, let $(f,K) \in \mathcal{B}_V$ be given and let $x \in K$ be a point satisfying $f(x) \in V$. Clearly, there is an open set $U \subseteq Y$ with $f(x) \in U \subseteq \overline{U} \subseteq V$. By the continuity of $f$, there is an open set $G \subseteq X$ such that $x \in G$ and $f(G) \subseteq U$. Let
\[ \mathcal{U}:=\big\lbrace g \in C(X,Y) \, ; \ g\big(\overline{G}\big) \subseteq V \big\rbrace \times \big\lbrace L \in \mathcal{K}(X) \, ; \ L \cap G \neq \emptyset \big\rbrace.\]
Again, it is easy to see that $\mathcal{U}$ is open, $\mathcal{U} \subseteq \mathcal{B}_V$ and $(f,K) \in \mathcal{U}$.
\end{proof}

The following lemma was essentially proved e.g. in \cite{Kennedy}. For the sake of completeness its proof is included.

\begin{lemma}\label{SpaceOfCtsFctionsEmbeddedIntoHyperspace}
Let $X$ be a compact metrizable space and $Y$ a Polish space. Then the mapping from $C(X,Y)$ to $\mathcal{K}(X \times Y)$ given by $f \mapsto \mathrm{graph}(f) = \big\lbrace (x,y) \in X \times Y ; \ y=f(x) \big\rbrace$ is a homeomorphic embedding.
\end{lemma}
\begin{proof}
Denote the mapping in question by $\Gamma$. Since $\Gamma$ is injective, it suffices to show that $\Gamma$ and $\Gamma^{-1}$ are continuous. Let $\varrho$ and $\sigma$ be arbitrary compatible metrics on $X$, $Y$ respectively. Define a metric $d$ on $X \times Y$ by $d \big( (x_1,y_1),(x_2,y_2) \big) = \max \big\lbrace \varrho(x_1,x_2),\sigma(y_1,y_2) \big\rbrace$, let $d_H$ be the Hausdorff metric on $\mathcal{K}(X \times Y)$ induced by $d$ and let $m$ be the uniform metric on $C(X,Y)$ induced by $\sigma$. That is,
\begin{align*}
    d_H (K,L) &= \max \Big\lbrace \max_{(x,y) \in K} \mathrm{dist}_d \big((x,y),L \big), \ \max_{(x,y) \in L} \mathrm{dist}_d \big( (x,y),K \big) \Big\rbrace,\\
    m(f,g) &= \max_{x \in X} \sigma \big( f(x),g(x) \big).
\end{align*}
Then $d$, $d_H$ and $m$ are compatible metrics on $X \times Y$, $\mathcal{K}(X \times Y)$ and $C(X,Y)$ respectively.
\begin{claim}\label{GammaIsLipschitz}
For any two functions $f,g \in C(X,Y)$, we have $d_H \big( \Gamma (f), \Gamma (g) \big) \leq m(f,g)$.
\end{claim}
\begin{claimproof}
Let $f,g \in C(X,Y)$. For every $(x,y) \in \Gamma (f)$, we have $y=f(x)$ and
\[ \mathrm{dist}_d \big((x,y), \Gamma(g) \big) \leq d \Big( (x,y), \big( x,g(x) \big) \Big) = \sigma \big( y,g(x) \big) = \sigma \big( f(x),g(x) \big) \leq m(f,g). \]
Similarly, $\mathrm{dist}_d \big((x,y), \Gamma(f) \big) \leq m(g,f)$ for every $(x,y) \in \Gamma (g)$. Thus, $d_H \big( \Gamma (f), \Gamma (g) \big) \leq m(f,g)$. \claimend
\end{claimproof}
The continuity of $\Gamma$ is an immediate consequence of Claim \ref{GammaIsLipschitz}. To show that $\Gamma^{-1}$ is continuous, we have to prove that for every $f \in C(X,Y)$ and $\varepsilon >0$, there is $\delta > 0$ such that $m(f,g) < \varepsilon$ for every $g \in C(X,Y)$ with $d_H \big( \Gamma (f), \Gamma (g) \big) < \delta$. Let $f \in C(X,Y)$ and $\varepsilon >0$ be given. By the (uniform) continuity of $f$, there is $\Delta >0$ such that $\sigma \big( f(x_1), f(x_2) \big) < \frac{1}{2} \varepsilon$ for every two points $x_1, x_2 \in X$ with $\varrho (x_1, x_2) < \Delta$. Let $\delta := \min \big\lbrace \Delta, \frac{1}{2} \varepsilon \big\rbrace$.
\begin{claim}
Assume that $g \in C(X,Y)$ satisfies $d_H \big( \Gamma (f), \Gamma (g) \big) < \delta$. Then $m(f,g) < \varepsilon$.
\end{claim}
\begin{claimproof}
Fix $x \in X$ with $m(f,g)=\sigma \big( f(x), g(x) \big)$. As $\big( x, g(x) \big) \in \Gamma (g)$ and $d_H \big( \Gamma (f), \Gamma (g) \big) < \delta$, there is $(z,y)\in \Gamma (f)$ such that $d \big( \big( x,g(x) \big), (z,y) \big) < \delta$. We have $\varrho (x,z) < \delta \leq \Delta$, $\sigma \big( y, g(x) \big) < \delta \leq \frac{1}{2} \varepsilon$ and $y=f(z)$. Thus, since $m(f,g) =\sigma \big( f(x), g(x) \big) \leq \sigma \big( f(x), f(z) \big) + \sigma \big( f(z), g(x) \big)$, it follows that $m(f,g) \leq \sigma \big( f(x), f(z) \big) + \sigma \big( y, g(x) \big) < \tfrac{1}{2}\varepsilon + \tfrac{1}{2}\varepsilon = \varepsilon$. \claimend
\end{claimproof}
\end{proof}

\begin{lemma}\label{CompactMapstoPreimageIsBorel}
Let $X$ be a compact metrizable space and $Y$ a Polish space. Then the mapping from $\big\lbrace (f,K) \in C(X,Y) \times \mathcal{K}(Y) \, ; \ f^{-1}(K) \neq \emptyset \big\rbrace$ to $\mathcal{K}(X)$ given by $(f,K) \mapsto f^{-1}(K)$ is Borel measurable.
\end{lemma}
\begin{proof}
Denote by $\mathcal{A}$ the domain of the mapping in question and define a mapping $\Psi \colon \mathcal{A} \to \mathcal{K}(X \times Y)$ by $\Psi (f,K) = \mathrm{graph}(f) \cap (X \times K)$. Then $\Psi$ is Borel measurable by Lemmata \ref{intersectionIsBorel} and \ref{SpaceOfCtsFctionsEmbeddedIntoHyperspace}. Denote by $\pi_1$ the coordinate projection from $X \times Y$ to $X$. By Lemma \ref{CtsFunctionAndCompactSetMapstoImageIsCts}, the mapping from $\mathcal{A}$ to $\mathcal{K}(X)$ given by $(f,K) \mapsto \pi_1 \big( \Psi (f,K) \big)$ is Borel measurable. However, $\pi_1 \big( \Psi (f,K) \big)$ is equal to $f^{-1}(K)$ for all $(f,K) \in \mathcal{A}$.
\end{proof}

The following lemma is implied by \cite[Theorem 3, §43]{Kuratowski}.

\begin{lemma}\label{ClosedSetInProductOfPolishAndCompact}
Let $X$ be a Polish space, $Y$ a compact metrizable space and $F \subseteq X \times Y$ a closed set. Then the mapping from $\big\lbrace x \in X \, ; \ \exists \, y \in Y : (x,y) \in F \big\rbrace$ to $\mathcal{K}(Y)$ given by $x \mapsto \big\lbrace y \in Y ; \, (x,y) \in F \big\rbrace$ is Borel measurable (in fact, it is upper semicontinuous).
\end{lemma}

\section{Compacta as continuous images of the Cantor space in a Borel measurable way}

In this section we prove that there exists a Borel measurable mapping $T \colon \mathcal{K}(Q) \to C(\mathcal{C},Q)$ such that for every $K \in \mathcal{K}(Q)$, the mapping $T(K)$ maps $\mathcal{C}$ onto $K$. This can be accomplished using the Kuratowski and Ryll-Nardzewski selection theorem. However, we shall present a more elegant approach. The basic idea is the following: We fix a suitable continuous surjection $\varphi \colon \mathcal{C} \to Q$. Then for any $K \in \mathcal{K}(Q)$, the restriction of $\varphi$ to $\varphi^{-1}(K)$ is a continuous mapping whose range is equal to $K$. Moreover, the assignment $K \mapsto \varphi \restriction_{\varphi^{-1}(K)}$ seems to be constructive enough to ensure Borel measurability. However, the problem is that the domain of $\varphi \restriction_{\varphi^{-1}(K)}$, i.e. the set $\varphi^{-1}(K)$, depends on $K$ and it is not equal to $\mathcal{C}$ (unless $K=Q$). The main tool allowing us to get rid of this problem is the following theorem due to Burgess (see \cite{Burgess}). We will refer to this theorem as the Burgess selection theorem.

\begin{theorem*}
Let $G$ be a Polish group, $X$ a Polish space and let $\alpha \colon G\times X \to X$ be a continuous action of $G$ on $X$. Denote by $E$ the orbit equivalence relation induced by $\alpha$ and let $Y$ be an $E$-invariant Borel subset of $X$. Let $E_Y$ be the restriction of $E$ to $Y$ and assume that $E_Y$ is countably separated. Then there is a Borel transversal for $E_Y$.
\end{theorem*}

Let us split the proof of the main theorem into three lemmata.

\begin{lemma}\label{PreimagesOfCompactaAreHomeoToCantor}
There exists a continuous surjective mapping $\varphi \colon \mathcal{C} \to Q$ such that $\varphi^{-1}(K)$ is homeomorphic to $\mathcal{C}$ for every $K \in \mathcal{K}(Q)$.
\end{lemma}
\begin{proof}
It is well-known that every nonempty compact metrizable space is a continuous image of $\mathcal{C}$. Therefore, there is a continuous surjection $\psi \colon \mathcal{C} \to Q$. Define a mapping $\Psi \colon \mathcal{C} \times \mathcal{C} \to Q$ by $\Psi(\alpha, \beta) = \psi (\alpha)$. Then $\Psi$ is a continuous surjection and $\Psi^{-1} (A) = \psi^{-1} (A) \times \mathcal{C}$ for every subset $A$ of $Q$. Hence, for every $K \in \mathcal{K}(Q)$, we can easily see that $\Psi^{-1}(K)$ is a nonempty, compact, zero-dimensional space with no isolated points. This shows that $\Psi^{-1} (K)$ is homeomorphic to $\mathcal{C}$ for every $K \in \mathcal{K}(Q)$. Finally, since $\mathcal{C} \times \mathcal{C}$ is homeomorphic to $\mathcal{C}$ as well, we can fix a homeomorphism $h \colon \mathcal{C} \to \mathcal{C} \times \mathcal{C}$ and then define $\varphi := \Psi \circ h$.
\end{proof}

Before stating the next lemma, let us denote
\[ \mathcal{K}_{\mathcal{C}}:= \big\lbrace K \in \mathcal{K}(\mathcal{C}) \, ; \ K \textup{ is homeomorphic to } \mathcal{C} \big\rbrace. \]

\begin{lemma}\label{BorelTransverzal}
There exists a Borel set $\mathcal{B} \subseteq \mathcal{E}(\mathcal{C}, \mathcal{C})$ such that for every $K \in \mathcal{K}_{\mathcal{C}}$, there is exactly one $f \in \mathcal{B}$ with $f (\mathcal{C})=K$. 
\end{lemma}
\begin{proof}
Denote $X:=\mathcal{E}(\mathcal{C}, \mathcal{C})$. Let $G$ be the group of self-homeomorphisms of $\mathcal{C}$, where the group operation is composition, and equip $G$ with the subspace topology inherited from $C(\mathcal{C},\mathcal{C})$.  Then $G$ is a Polish group (this is well-known, see e.g. \cite[Example 2.2.4]{Gao}). Define a mapping $\alpha \colon G \times X \to X$ by $\alpha (g,f)=f\circ g$. Then $\alpha$ is an action of $G$ on $X$ and it is easy to see that $\alpha$ is continuous. Denote by $E$ the orbit equivalence relation induced by $\alpha$. Clearly, for all $f_1, f_2 \in X$, we have
\[ f_1 E f_2 \iff f_1(\mathcal{C}) = f_2(\mathcal{C}). \]
Thus, to prove this lemma, it suffices to show that there is a Borel transversal for $E$. To that end, we are going to use the Burgess selection theorem with $Y:=X$. We have already verified most of the assumptions of the selection theorem, it remains to show that $E$ is countably separated. Let $(U_n)_{n=1}^{\infty}$ be a sequence of open subset of $\mathcal{C}$ forming a base for the topology of $\mathcal{C}$. For $n \in \N$, let
\[ Z_n := \big\lbrace f \in X \, ; \ U_n \cap f(\mathcal{C}) \neq \emptyset \big\rbrace . \]
Then for any $n \in \N$, the set $Z_n$ is $E$-invariant and open in $X$. Moreover, for all $f_1, f_2 \in X$, we have
\[ f_1 E f_2 \iff f_1(\mathcal{C}) = f_2(\mathcal{C}) \iff \{ n \in \N \, ; \ f_1 \in Z_n \} = \{ n \in \N \, ; \ f_2 \in Z_n \} . \]
\end{proof}

\begin{lemma}\label{CantorImagesLemma}
There is a Borel measurable mapping $T_0 \colon \mathcal{K}_{\mathcal{C}} \to \mathcal{E}(\mathcal{C}, \mathcal{C})$ such that for every $K \in \mathcal{K}_{\mathcal{C}}$, the image of $\mathcal{C}$ under $T_0(K)$ is equal to $K$. 
\end{lemma}
\begin{proof}
Let $\mathcal{B} \subseteq \mathcal{E}(\mathcal{C}, \mathcal{C})$ be the Borel set given by Lemma \ref{BorelTransverzal}. Define a mapping $\Phi \colon \mathcal{B} \to \mathcal{K}(\mathcal{C})$ by $\Phi(f)=f(\mathcal{C})$. Lemma \ref{CtsFunctionAndCompactSetMapstoImageIsCts} implies that $\Phi$ is continuous. In particular, $\Phi$ is Borel measurable. Moreover, $\Phi$ is injective and $\Phi(\mathcal{B})=\mathcal{K}_{\mathcal{C}}$. Hence, $\mathcal{K}_{\mathcal{C}}$ is Borel and the mapping $\Phi^{-1} \colon \mathcal{K}_{\mathcal{C}} \to \mathcal{E}(\mathcal{C}, \mathcal{C})$ is Borel measurable (see, e.g., \cite[Corollary 15.2]{Kechris}). Let $T_0 := \Phi^{-1}$.
\end{proof}

\begin{theorem}\label{CantorImagesThm}
There is a Borel measurable mapping $T \colon \mathcal{K}(Q) \to C(\mathcal{C},Q)$ such that for every $K \in \mathcal{K}(Q)$, the image of $\mathcal{C}$ under $T(K)$ is equal to $K$.
\end{theorem}

\begin{proof}
Let $\varphi \colon \mathcal{C} \to Q$ and $T_0 \colon \mathcal{K}_{\mathcal{C}} \to \mathcal{E}(\mathcal{C}, \mathcal{C})$ be the mappings guaranteed by Lemmata \ref{PreimagesOfCompactaAreHomeoToCantor} and \ref{CantorImagesLemma} respectively. Define mappings $\Theta \colon \mathcal{E}(\mathcal{C}, \mathcal{C}) \to C(\mathcal{C},Q)$ and $\Psi \colon \mathcal{K}(Q) \to \mathcal{K}(\mathcal{C})$ by $\Theta (f) = \varphi \circ f$ and $\Psi (K) = \varphi^{-1}(K)$. It is clear that $\Theta$ is continuous and, by Lemma \ref{CompactMapstoPreimageIsBorel}, $\Psi$ is Borel measurable. Also, note that $\Psi \big( \mathcal{K}(Q) \big) \subseteq \mathcal{K}_{\mathcal{C}}$. Finally, define $T:= \Theta \circ T_0 \circ \Psi$. Then $T$ is Borel measurable and it is easy to see that for every $K \in \mathcal{K}(Q)$, the image of $\mathcal{C}$ under $T(K)$ is equal to $\varphi \big( \varphi^{-1}(K) \big) =K$.
\end{proof}

\section{The Hahn-Mazurkiewicz Theorem in a Borel\\measurable way}

The aim of this section is to prove the existence of a Borel measurable mapping $\Phi \colon \mathsf{LC} (Q) \to C(I,Q)$ such that for every $K \in \mathsf{LC} (Q)$, the image of $I$ under $\Phi (K)$ is equal to $K$. Our basic strategy is to go through the proof of the Hahn-Mazurkiewicz theorem and verify that every step can be carried out in a Borel fashion.

The following definition and theorem can be found in \cite[p. 120]{Nadler}

\begin{definition}\label{DefinitionOfPropertyS}
A nonempty subset $X$ of a metric space $Y$ is said to have property S provided that for each $\varepsilon > 0$ there is $n \in \N$ and connected sets $A_1, \dotsc , A_n \subseteq Y$ such that $A_1 \cup \dots \cup A_n = X$ and $\mathrm{diam} (A_i) < \varepsilon$ for every $i = 1, \dotsc , n$.
\end{definition}

\begin{theorem}\label{LCiffPropertyS}
A nonempty compact subset of a metric space is locally connected if and only if it has property S.
\end{theorem}

\begin{corollary}\label{compactLCiffPropertyS}
A nonempty compact subset $X$ of a metric space $Y$ is locally connected if and only if for each $\varepsilon > 0$ there is $n \in \N$ and continua $K_1, \dotsc , K_n \subseteq Y$ such that $K_1 \cup \dots \cup K_n = X$ and $\mathrm{diam} (K_i) < \varepsilon$ for every $i = 1, \dotsc , n$.
\end{corollary}

\begin{proof}
The ``if" part immediately follows from Theorem \ref{LCiffPropertyS}. To prove the ``only if" part assume that $X$ is locally connected. By Theorem \ref{LCiffPropertyS}, $X$ has property S. Let $\varepsilon > 0$ be arbitrary and let $A_1, \dotsc , A_n \subseteq Y$ be the corresponding connected sets given by the property S. For each $i = 1, \dotsc , n$, let $K_i = \overline{A_i}$.
\end{proof}

Let us denote by $\mathcal{Z}$ the Polish space of all finite sequences of continua in $Q$. Formally speaking, $\mathcal{Z}$ is the topological sum
\[ \bigoplus_{n \in \N} \big( \mathsf{C} (Q) \big)^n . \]
Since $Q$ is compact, so is $\mathcal{K}(Q)$. Thus, being a closed subspace of $\mathcal{K}(Q)$, the space $\mathsf{C} (Q)$ is compact and so are its powers. Therefore, $\mathcal{Z}$ is $\sigma$-compact.

Let $\varrho$ be a compatible metric on $Q$. This metric will be fixed from now on (in this chapter).

\begin{theorem}\label{PeanoContinuumCanBeCoveredBySmallContinua}
 There is a Borel measurable mapping $\Psi \colon \mathsf{LC} (Q) \times {\R}^+ \to \mathcal{Z}$ such that for every $X \in \mathsf{LC} (Q)$ and $\varepsilon \in {\R}^+$, if $\, \Psi (X, \varepsilon) = (K_1 , \dotsc , K_n)$, then $K_1 \cup \dots \cup K_n = X$ and $\mathrm{diam}_{\varrho} (K_i) < \varepsilon$ for each $i = 1, \dotsc , n$.
\end{theorem}
\begin{proof}
For every $\Gamma \in \mathcal{Z}$, there is a unique natural number $l(\Gamma)$ such that $\Gamma \in \big( \mathsf{C} (Q) \big)^{l(\Gamma)}$. Using this notation, define a mapping $f \colon \mathcal{Z} \to \mathcal{K} (Q)$ by
\[ f(\Gamma) = \bigcup_{i=1}^{l(\Gamma)} \Gamma_i .\]
Since $\big( \mathsf{C} (Q) \big)^k$ is an open subset of $\mathcal{Z}$ for every $k \in \N$, it is easy to see that $f$ is continuous. Define a function $g \colon \mathsf{C} (Q) \to \R$ by $g(K)=\mathrm{diam}_{\varrho} (K)$. Clearly, $g$ is continuous as well. Finally, define
\[ \mathcal{A} := \Big\lbrace (X , \varepsilon , \Gamma) \in \mathsf{C} (Q) \times {\R}^+ \times \mathcal{Z} \, ; \ f(\Gamma)  = X , \ g ( \Gamma_i ) \leq \varepsilon \textup{ for } i = 1, \dotsc , l(\Gamma) \Big\rbrace . \]
The continuity of $f$ and $g$ shows that $\mathcal{A}$ is a closed subset of the Polish space $\mathsf{C} (Q) \times {\R}^+ \times \mathcal{Z}$. In particular, $\mathcal{A}$ is a Borel set with $\sigma$-compact sections in $\mathcal{Z}$. By the Arsenin-Kunugui selection theorem \cite[Theorem 35.46]{Kechris}, the set $\pi (\mathcal{A}) := \big\lbrace (X, \varepsilon) \in \mathsf{C} (Q) \times {\R}^+ \, ; \ \exists \, \Gamma \in \mathcal{Z} : (X , \varepsilon , \Gamma) \in \mathcal{A} \big\rbrace$ is Borel in $\mathsf{C} (Q) \times {\R}^+$ and there exists a Borel measurable map $\widehat{\Psi} \colon \pi (\mathcal{A}) \to \mathcal{Z}$ such that $\big( X, \varepsilon , \widehat{\Psi} (X, \varepsilon) \big) \in \mathcal{A}$ for every $(X, \varepsilon) \in \pi (\mathcal{A})$. By Corollary \ref{compactLCiffPropertyS}, we have $\mathsf{LC} (Q) \times {\R}^+ \subseteq \pi (\mathcal{A})$. The desired mapping $\Psi \colon \mathsf{LC} (Q) \times {\R}^+ \to \mathcal{Z}$ can therefore be defined by $\Psi (X, \varepsilon) = \widehat{\Psi} (X, \frac{\varepsilon}{2})$.
\end{proof}

The very definition of property S clearly implies the next proposition.

\begin{proposition}\label{ClosureOfPropertyS}
Let $X$ be a nonempty subset of a metric space $Y$. If $X$ has property S, then so does the closure of $X$.
\end{proposition}

Define $\mathcal{R}:= \big\lbrace (X,A,\varepsilon) \in \mathsf{LC} (Q) \times \mathsf{C} (Q) \times {\R}^+ ; \ A \subseteq X \big\rbrace$ and for every $(X,A,\varepsilon) \in \mathcal{R}$, let
\begin{align*}
    S(X,A,\varepsilon) := \big\lbrace x \in X \, ; \ \ &\exists \, n \in \N \ \, \exists \, D_1, \dotsc ,D_n \subseteq X: \ x \in D_n \, , \ A \cap D_1 \neq \emptyset \, ,\\
    &\forall \, i \in \{ 2, \dotsc , n \} : D_{i-1} \cap D_i \neq \emptyset \, ,\\
    &\forall \, i \in \{ 1, \dotsc , n \} : D_i \textup{ is connected},\\
    &\forall \, i \in \{ 1, \dotsc , n \} : \mathrm{diam}_{\varrho} (D_i) < \varepsilon \cdot 2^{-i} \big\rbrace .
\end{align*}
For every $(X,A,\varepsilon) \in \mathcal{R}$, it easily follows from the compactness of $X$ that
\begin{align*}
    S(X,A,\varepsilon) = \big\lbrace x \in X \, ; \ \ &\exists \, n \in \N \ \, \exists \, K_1, \dotsc ,K_n \in \mathsf{C}(Q):\\
    &x \in K_n \, , \ A \cap K_1 \neq \emptyset \, ,\\
    &\forall \, i \in \{ 2, \dotsc , n \} : K_{i-1} \cap K_i \neq \emptyset \, ,\\
    &\forall \, i \in \{ 1, \dotsc , n \} : K_i \subseteq X ,\\
    &\forall \, i \in \{ 1, \dotsc , n \} : \mathrm{diam}_{\varrho} (K_i) < \varepsilon \cdot 2^{-i} \big\rbrace .
\end{align*}
The following proposition is a consequence of \cite[8.7, 8.8]{Nadler} combined with Theorem \ref{LCiffPropertyS} and Proposition \ref{ClosureOfPropertyS}.

\begin{proposition}\label{SXAepsilonIsLC}
For each $(X,A,\varepsilon) \in \mathcal{R}$, the set $\overline{S(X,A,\varepsilon)}$ is a Peano continuum and we have $\mathrm{diam}_{\varrho} \big( S(X,A,\varepsilon) \big) \leq 2 \varepsilon + \mathrm{diam}_{\varrho} (A)$.
\end{proposition}

\begin{proposition}\label{BorelSXAepsilon}
The mapping $\Theta \colon \mathcal{R} \to \mathsf{LC} (Q)$ defined by
\[\Theta (X,A,\varepsilon) = \overline{S(X,A,\varepsilon)}\]
is Borel measurable.
\end{proposition}
\begin{proof}
For every $n \in \N$ and $(X,A,\varepsilon) \in \mathcal{R}$, let
\begin{align*}
    S_n(X,A,\varepsilon) = \big\lbrace x \in X \, ; \ \ &\exists \, K^1, \dotsc ,K^n \in \mathsf{C}(Q): \, x \in K^n \, , \ A \cap K^1 \neq \emptyset,\\
    &\forall \, i \in \{ 2, \dotsc , n \} : K^{i-1} \cap K^i \neq \emptyset,\\
    &\forall \, i \in \{ 1, \dotsc , n \} : K^i \subseteq X ,\\
    &\forall \, i \in \{ 1, \dotsc , n \} : \mathrm{diam}_{\varrho} (K^i) \leq \varepsilon \cdot 2^{-i} \big\rbrace .
\end{align*}
\begin{claim}\label{SnXAepsilonIsCompact}
For every $n \in \N$ and $(X,A,\varepsilon) \in \mathcal{R}$, the set $S_n(X,A,\varepsilon)$ is compact.
\end{claim}
\begin{claimproof}
Since $X$ is metrizable compact, it suffices to show that $S_n(X,A,\varepsilon)$ is sequentially closed in $X$. Let $(x_k)_{k=1}^{\infty}$ be a sequence in $S_n(X,A,\varepsilon)$ converging to $x \in X$. For all $k \in \N$, let $K_k^1,\dotsc ,K_k^n \in \mathsf{C}(Q)$ be continua witnessing that $x_k \in S_n(X,A,\varepsilon)$. Since $\mathsf{C}(Q)$ is metrizable compact, we can assume that there exist $K^1,\dotsc ,K^n \in \mathsf{C}(Q)$ such that $\lim_{k \to \infty} K_k^i = K^i$ for $i=1,\dotsc ,n$. It is straightforward to verify that the continua $K^1,\dotsc ,K^n$ witness that $x \in S_n(X,A,\varepsilon)$.\claimend
\end{claimproof}
For every $n \in \N$, Claim \ref{SnXAepsilonIsCompact} shows that $S_n$ is a mapping from $\mathcal{R}$ to $\mathcal{K}(Q)$. It is easy to see that for each $(X,A,\varepsilon) \in \mathcal{R}$,
\[ S(X,A,\varepsilon) = \bigcup_{k,n \in \N} S_n \Big( \big( X,A, \tfrac{k}{k+1} \varepsilon \big) \Big), \]
therefore,
\[ \Theta (X,A,\varepsilon) = \overline{\bigcup_{k,n \in \N} S_n \Big( \big( X,A, \tfrac{k}{k+1} \varepsilon \big) \Big)}. \]
By Lemma \ref{closureOfUnionIsBorel}, this proposition will be proved once we show that each of the mappings $S_1, S_2, S_3, \dotsc$ is Borel measurable. Thus, let $n \in \N$ be fixed and let $F \subseteq Q$ be an arbitrary closed set. By Lemma \ref{CharacterizationOfBorelnessIntoVietoris}, it suffices to show that the set $\mathcal{F}:= \big\lbrace (X,A,\varepsilon) \in \mathcal{R} \, ; \ F \cap S_n(X,A,\varepsilon) \neq \emptyset \big\rbrace$ is Borel in $\mathcal{R}$. We claim that this set is actually closed (relatively) in $\mathcal{R}$. Let $\big( (X_k,A_k,\varepsilon_k) \big)_{k \in \N}$ be a sequence in $\mathcal{F}$ converging to $(X,A,\varepsilon) \in \mathcal{R}$, we are going to show that $(X,A,\varepsilon) \in \mathcal{F}$. For every $k \in \N$, fix a point $x_k \in F \cap S_n (X_k,A_k,\varepsilon_k)$ and let $K_k^1,\dotsc ,K_k^n \in \mathsf{C}(Q)$ witness that $x_k \in S_n (X_k,A_k,\varepsilon_k)$. Since $Q$ and $\mathsf{C}(Q)$ are metrizable and compact, we can assume that there exist $K^1,\dotsc ,K^n \in \mathsf{C}(Q)$ and $x \in Q$ such that $\lim_{k \to \infty} x_k = x$ and $\lim_{k \to \infty} K_k^i = K^i$ for every $i=1,\dotsc ,n$. Then $x \in F$ and it is straightforward to verify that $K^1,\dotsc ,K^n$ witness that $x \in S_n(X,A,\varepsilon)$. Thus, $(X,A,\varepsilon) \in \mathcal{F}$.
\end{proof}

Denote by $\mathcal{Z}_P$ the subspace of $\mathcal{Z}$ consisting of finite sequences of Peano continua. Formally,
\[ \mathcal{Z}_P = \big\lbrace \Gamma \in \mathcal{Z} \, ; \ \textup{if } \Gamma = (K_1,\dotsc ,K_n) \textup{, then } K_i \in \mathsf{LC}(Q) \textup{ for } i=1,\dotsc , n \big\rbrace. \]
The following theorem is a strengthening of Theorem \ref{PeanoContinuumCanBeCoveredBySmallContinua} and it is a key ingredient for the proof of the Borel version of the Hahn-Mazurkiewicz theorem.

\begin{theorem}\label{PeanoContinuumCanBeCoveredBySmallPeanoContinua}
There is a Borel measurable mapping $\Psi_P \colon \mathsf{LC} (Q) \times {\R}^+ \to \mathcal{Z}_P$ such that for every $X \in \mathsf{LC} (Q)$ and $\varepsilon \in {\R}^+$, if $\, \Psi_P (X, \varepsilon) = (K_1 , \dotsc , K_n)$, then $K_1 \cup \dots \cup K_n = X$ and $\mathrm{diam}_{\varrho} (K_i) < \varepsilon$ for each $i = 1, \dotsc , n$.
\end{theorem}
\begin{proof}
Let $\Psi \colon \mathsf{LC} (Q) \times {\R}^+ \to \mathcal{Z}$ be the mapping provided by Theorem \ref{PeanoContinuumCanBeCoveredBySmallContinua} and let $\Theta \colon \mathcal{R} \to \mathsf{LC} (Q)$ be the mapping studied in Proposition \ref{BorelSXAepsilon}. Define the desired mapping $\Psi_P \colon \mathsf{LC} (Q) \times {\R}^+ \to \mathcal{Z}_P$ in the following way: For any $(X, \varepsilon) \in \mathsf{LC} (Q) \times {\R}^+$, if $\Psi \big( X, \tfrac{1}{3} \varepsilon \big) = (K_1, \dotsc ,K_n)$, let
\[ \Psi_P (X, \varepsilon) = \Big( \Theta \big( X,K_1, \tfrac{1}{3} \varepsilon \big) , \dotsc ,\Theta \big( X,K_n, \tfrac{1}{3} \varepsilon \big) \Big) . \]
Since $\Psi \big( X, \tfrac{1}{3} \varepsilon \big) = (K_1, \dotsc ,K_n)$, each of the sets $K_1, \dotsc ,K_n$ is a continuum contained in $X$ with diameter less than $\frac{1}{3} \varepsilon$. Thus, by Proposition \ref{SXAepsilonIsLC}, each of the sets $\Theta \big( X,K_1, \tfrac{1}{3} \varepsilon \big) , \dotsc ,\Theta \big( X,K_n, \tfrac{1}{3} \varepsilon \big)$ is a Peano continnum with diameter less than $\varepsilon$. Moreover, it is clear that $K_i \subseteq S \big( X,K_i,\frac{1}{3} \varepsilon \big) \subseteq X$ for every $i=1, \dotsc ,n$. Therefore, as $\Psi \big( X, \tfrac{1}{3} \varepsilon \big) = (K_1, \dotsc ,K_n)$, we have
\[ X = \bigcup_{i=1}^n K_i \subseteq \bigcup_{i=1}^n S \big( X,K_i,\tfrac{1}{3} \varepsilon \big) \subseteq X, \]
which implies that
\[ \bigcup_{i=1}^n \Theta \big( X,K_i, \tfrac{1}{3} \varepsilon \big) = \bigcup_{i=1}^n \overline{S \big( X,K_i,\tfrac{1}{3} \varepsilon \big)} = X. \]
Finally, as both of the mappings $\Psi$ and $\Theta$ are Borel measurable, $\Psi_P$ is Borel measurable too.
\end{proof}
The following lemma is a special case of \cite[8.13]{Nadler}.
\begin{lemma}\label{SequenceOfContinuaIntoChain}
Let $n \in \N$ and let $K_1, \dotsc ,K_n \subseteq Q$ be continua such that the set $X:= K_1 \cup \dots \cup K_n$ is connected. Then for any two points $x,y\in X$, there is $m \in \N$ and continua $L_1, \dotsc ,L_m \subseteq Q$ such that $x \in L_1$, $y \in L_m$, $\{ K_1, \dotsc ,K_n \} = \{ L_1, \dotsc ,L_m \}$ and $L_{i-1} \cap L_i \neq \emptyset$ for every $i \in \{ 2, \dotsc ,m \}$.
\end{lemma}
Now, let us present a Borel version of Lemma \ref{SequenceOfContinuaIntoChain}. Define
\begin{align*}
    \mathcal{M} := \Big\lbrace (\Gamma ,x,y) \in \mathcal{Z} \times Q \times Q \, ; \ &\textup{if } \Gamma = (K_i)_{i=1}^n \textup{, then the set } K_1 \cup \dots \cup K_n\\
    &\textup{is connected and contains both } x \textup{ and } y \Big\rbrace
\end{align*}
and equip $\mathcal{M}$ with the subspace topology inherited from $\mathcal{Z} \times Q \times Q$.
\begin{lemma}\label{BorelSequenceOfContinuaIntoChain}
There is a Borel measurable mapping $\zeta \colon \mathcal{M} \to \mathcal{Z}$ such that for every $(\Gamma ,x,y) \in \mathcal{M}$, if $\Gamma = (K_1, \dotsc ,K_n)$ and $\zeta (\Gamma ,x,y) = (L_1, \dotsc ,L_m)$, then $\{ K_1, \dotsc ,K_n \} = \{ L_1, \dotsc ,L_m \}$, $x \in L_1$, $y \in L_m$ and $L_{i-1} \cap L_i \neq \emptyset$ for every $i=2, \dotsc ,m$.
\end{lemma}
\begin{proof}
For a given $(\Gamma ,x,y) \in \mathcal{M}$ with $\Gamma = (K_1, \dotsc ,K_n)$, the task of finding the corresponding finite sequence $(L_1, \dotsc ,L_m) \in \mathcal{Z}$ is equivalent to the task of finding a finite sequence $(k_1, \dotsc ,k_m)$ of natural numbers such that $\{ k_1, \dotsc ,k_m \} = \{ 1, \dotsc ,n \}$, $x \in K_{k_1}$, $y \in K_{k_m}$ and $K_{k_{i-1}} \cap K_{k_i} \neq \emptyset$ for $i=2, \dotsc ,m$. Indeed, once these numbers $k_1, \dotsc ,k_m$ are found, we can simply define $L_i := K_{k_i}$ for each $i=1, \dotsc ,m$. Therefore, the task of finding the Borel measurable mapping $\zeta$ is closely related to the task of finding a suitable Borel measurable mapping from $\mathcal{M}$ to a suitable space of finite sequences of natural numbers. Formally, we are going to represent finite sequences of natural number by infinite sequences of nonnegative integers whose terms are equal to zero from some point on. Define
\[ T:= \Big\lbrace \alpha = (\alpha_1, \alpha_2, \dotsc) \in {\N_0}^{\N} \, ; \ \exists \, m,n \in \N : \{ \alpha_1, \dotsc , \alpha_m \} = \{ 1, \dotsc ,n \} \, , \ \alpha_{m+1} = \alpha_{m+2} = \dots = 0 \Big\rbrace \]
and equip this set with the discrete topology. For every $\alpha \in T$, let $m(\alpha), n(\alpha) \in \N$ be the numbers witnessing that $\alpha \in T$. For every $\Gamma \in \mathcal{Z}$, let $l(\Gamma) \in \N$ be the number satisfying $\Gamma \in \big( \mathsf{C} (Q) \big)^{l(\Gamma)}$. For every $(\Gamma ,x,y) \in \mathcal{M}$, let
\[ A(\Gamma ,x,y) := \Big\lbrace \alpha \in T \, ; \, \ l(\Gamma)=n(\alpha), \ x \in \Gamma_{\alpha_1}, \ y \in \Gamma_{\alpha_{m(\alpha)}}, \ \Gamma_{\alpha_{i-1}} \cap \Gamma_{\alpha_i} \neq \emptyset \, \textup{ for } i=2, \dotsc ,m(\alpha) \Big\rbrace. \]
By Lemma \ref{SequenceOfContinuaIntoChain}, the set $A(\Gamma ,x,y)$ is nonempty for each $(\Gamma ,x,y) \in \mathcal{M}$. Also, it is easy to see that for every $\alpha \in T$, the set $\mathcal{M}_{\alpha} := \big\lbrace (\Gamma ,x,y) \in \mathcal{M} \, ; \ \alpha \in A(\Gamma ,x,y) \big\rbrace$ is closed in $\mathcal{M}$. As $T$ is countable, we can write $T = \big\lbrace \alpha^1 , \alpha^2 , \alpha^3 , \dotsc \big\rbrace$. Define mappings $\mu \colon \mathcal{M} \to \N$ and $\zeta_1 \colon \mathcal{M} \to T$ by
\[ \mu (\Gamma ,x,y) = \mathrm{min} \big\lbrace k \in \N \, ; \ \alpha^k \in A(\Gamma ,x,y) \big\rbrace \, , \ \ \ \zeta_1 (\Gamma ,x,y) = \alpha^{\mu (\Gamma ,x,y)}. \]
Let us show that $\mu$ is Borel measurable. Since $\N$ is a discrete space, we have to show that $\mu^{-1} (\{ j \})$ is Borel in $\mathcal{M}$ for each $j \in \N$. Clearly, $\mu^{-1} (\{ 1 \}) = \big\lbrace (\Gamma ,x,y) \in \mathcal{M} \, ; \ \alpha^1 \in A(\Gamma ,x,y) \big\rbrace = \mathcal{M}_{\alpha^1}$ and for every $j \in \N$, we have $\mu^{-1} (\{ j+1 \}) = \mathcal{M}_{\alpha^{j+1}} \setminus \big( \mathcal{M}_{\alpha^1} \cup \dots \cup \mathcal{M}_{\alpha^j} \big)$.
Thus, $\mu$ is Borel measurable and it follows that $\zeta_1$ is Borel measurable too.

Finally, let $\mathcal{N}:= \big\lbrace (\Gamma, \alpha) \in \mathcal{Z} \times T \, ; \ n(\alpha) \leq l(\Gamma) \big\rbrace$ and define a mapping $\zeta_2 \colon \mathcal{N} \to \mathcal{Z}$ by $\zeta_2 (\Gamma, \alpha) = \big( \Gamma_{\alpha_1}, \dotsc ,\Gamma_{\alpha_{m(\alpha)}} \big)$. Clearly, $\zeta_2$ is continuous. Define the desired mapping $\zeta \colon \mathcal{M} \to \mathcal{Z}$ by $\zeta (\Gamma,x,y) = \zeta_2 \big( \Gamma, \zeta_1 (\Gamma ,x,y) \big)$. Since both of the mappings $\zeta_1$, $\zeta_2$ are Borel measurable, so is $\zeta$. It is straightforward to verify that $\zeta$ is the mapping we are after.
\end{proof}

For any topological space $Y$, denote by $\mathcal{F} (Y)$ the family of all nonempty closed subsets of $Y$. Recall that for any two topological spaces $X$ and $Y$, a mapping $f \colon X \to \mathcal{F} (Y)$ is said to be upper semi-continuous if the set $\big\lbrace x \in X \, ; \ f(x) \subseteq V \big\rbrace$ is open in $X$ for every open set $V \subseteq Y$.
The following theorem can be found in \cite[7.4]{Nadler}.

\begin{theorem}\label{GeneralMapping}
 Let $X$ and $Y$ be nonempty compact metric spaces and let $f_n \colon X \to \mathcal{F} (Y)$, $n \in \N$, be upper semi-continuous mappings such that:
 \begin{itemize}
     \item $\forall \, n \in \N \ \forall \, x \in X : f_{n+1} (x) \subseteq f_n (x)$;
     \item $\forall \, n \in \N : \ \displaystyle \bigcup \big\lbrace f_n (x) \, ; \ x \in X \big\rbrace = Y $;
     \item $\forall \, x \in X : \displaystyle \lim_{n \to \infty} \mathrm{diam} \big( f_n (x) \big) = 0$.
 \end{itemize}
Then the (uniquely determined) mapping $f \colon X \to Y$ satisfying 
\[ \big\lbrace f(x) \big\rbrace = \bigcap_{n \in \N} f_n (x) \]
for each $x \in X$ is surjective and continuous.
\end{theorem}

Denote by ${\N}^{< \N}$ the set of all (nonempty) finite sequences of natural numbers and equip ${\N}^{< \N}$ with the discrete topology. More formally, ${\N}^{< \N}$ is the topological sum
\[ \bigoplus_{p \in \N}{\N}^p,\]
where $\N$ is considered as a discrete space.

\begin{definition}
Let $X \in \mathsf{C}(Q)$, $\varepsilon > 0$ and $\Gamma = (K_1, \dotsc , K_k) \in \mathcal{Z}_P$. We say that $\Gamma$ is a weak $\varepsilon$-chain covering $X$ if the following three conditions hold:
\begin{enumerate}[label=(\roman*),font=\textup, noitemsep]
    \item $K_1 \cup \dots \cup K_k = X$;
    \item $\mathrm{diam}_{\varrho} (K_i) < \varepsilon$ for $i = 1, \dotsc ,k \,$;
    \item $K_{i-1} \cap K_i \neq \emptyset$ for $i = 2, \dotsc ,k \,$.
\end{enumerate}
\end{definition}

\begin{definition}
Let $\Gamma = (K_1, \dotsc , K_k) \in \mathcal{Z}$, $\Gamma' = (L_1, \dotsc , L_l) \in \mathcal{Z}$ and let $\gamma = (j_1, \dotsc , j_m) \in {\N}^{< \N}$. We say that $\Gamma'$ is a refinement of $\Gamma$ coded by $\gamma$ provided that $m=k$ and the following three assertions are satisfied:
\begin{enumerate}[label=(\roman*),font=\textup, noitemsep]
    \item $1=j_1 < j_2 < \dots < j_k \leq l \,$;
    \item $\ L_{j_k} \cup L_{j_k +1} \cup \dots \cup L_l = K_k \,$;
    \item $L_{j_{i-1}} \cup L_{j_{i-1}+1} \cup \dots \cup L_{j_i-1} = K_{i-1}$ for $i = 2, \dotsc ,k \,$.
\end{enumerate}
\end{definition}

\begin{lemma}\label{BorelSequenceOfSubChains}
There exist Borel measurable mappings $\mu_n \colon \mathsf{LC}(Q) \to {\N}^{< \N}$ and $\Psi_n \colon \mathsf{LC}(Q) \to \mathcal{Z}_P$, $n \in \N$, such that for every $X \in \mathsf{LC}(Q)$ and every $n \in \N$, the following two assertions hold:
\begin{enumerate}[label=(\roman*),font=\textup]
    \item $\Psi_n (X)$ is a weak $2^{-n}$-chain covering $X$.
    \item If $n > 1$, then $\Psi_n (X)$ is a refinement of $\Psi_{n-1} (X)$ coded by $\mu_n (X)$.
\end{enumerate}
\end{lemma}
\begin{proof}
Let $\Psi_P \colon \mathsf{LC} (Q) \times {\R}^+ \to \mathcal{Z}_P$ and $\zeta \colon \mathcal{M} \to \mathcal{Z}$ be the mappings given by Theorem \ref{PeanoContinuumCanBeCoveredBySmallPeanoContinua} and Lemma \ref{BorelSequenceOfContinuaIntoChain} respectively. Let $\sigma \colon \mathcal{K} (Q) \to Q$ be a Borel measurable mapping satisfying $\sigma (X) \in X$ for each $X \in \mathcal{K} (Q)$. It is well-known that such a mapping exists (it is a simple application of the Kuratowski and Ryll-Nardzewski selection theorem). Define $\Psi_1 \colon \mathsf{LC}(Q) \to \mathcal{Z}_P$ by
\[ \Psi_1 (X) = \zeta \Big( \Psi_P \big( X, \tfrac{1}{2} \big), \sigma (X), \sigma (X) \Big) \]
and let $\mu_1 \colon \mathsf{LC}(Q) \to {\N}^{< \N}$ be an arbitrary Borel measurable mapping. It is easy to see that $\Psi_1$ is Borel measurable and the assertion (i) is satisfied for $n=1$. Note that the assertion (ii) does not say anything about $n=1$. Let us proceed by induction. Assume that $n \in \N \setminus \{ 1 \}$ is given and the mappings $\mu_1, \dotsc, \mu_{n-1}$ and $\Psi_1, \dotsc, \Psi_{n-1}$ have already been found. Let $X \in \mathsf{LC} (Q)$ be given, we are going to define $\Psi_n (X)$ and $\mu_n (X)$ in the following way: Assume that $\Psi_{n-1} (X)=(K_1, \dotsc ,K_k)$. If $k=1$, then (since $K_1 \cup \dots \cup K_k = X$) we have $K_1 = X$ and we can simply define
\[ \mu_n (X) = (1) \, , \ \Psi_n (X) = \zeta \Big( \Psi_P \big( X, 2^{-n} \big), \sigma (X), \sigma (X) \Big). \]
Let us focus on the case when $k>1$. Assume that
\begin{align*}
    \zeta \Big( \Psi_P \big( K_1, 2^{-n} \big), \sigma (K_1), \sigma (K_1 \cap K_2) \Big) &= \big( L_1^1, \dotsc ,L_{l(1)}^1 \big),\\
    \zeta \Big( \Psi_P \big( K_k, 2^{-n} \big), \sigma (K_{k-1} \cap K_k), \sigma (K_k) \Big) &= \big( L_1^k, \dotsc ,L_{l(k)}^k \big)
\end{align*}
and
\[ \zeta \Big( \Psi_P \big( K_i, 2^{-n} \big), \sigma (K_{i-1} \cap K_i), \sigma (K_i \cap K_{i+1}) \Big) = \big( L_1^i, \dotsc ,L_{l(i)}^i \big) \]
for every $i \in \N$ with $1<i<k$. Then we can define
\begin{align*}
    \Psi_n (X) &= \big( L_1^1, \dotsc ,L_{l(1)}^1, L_1^2, \dotsc L_{l(2)}^2, \dotsc , L_1^k, \dotsc ,L_{l(k)}^k \big),\\
    \mu_n (X) &= \big( 1, 1+l(1), \dotsc , 1+l(1)+\dots +l(k-1) \big) = \bigg( 1+\sum_{j=1}^{i-1} l(j) \bigg)_{i=1}^{k}.
\end{align*}
It is straightforward (although quite tedious) to verify that $\Psi_n$ and $\mu_n$ are Borel measurable (among other things, Lemma \ref{intersectionIsBorel} is used here) and both of the assertions (i) and (ii) are satisfied for $n$.
\end{proof}

\begin{theorem}\label{MainTheorem}
There is a Borel measurable mapping $\Phi \colon \mathsf{LC}(Q) \to C(I,Q)$ such that for every $X \in \mathsf{LC}(Q)$, the image of $I$ under $\Phi (X)$ is equal to $X$.
\end{theorem}
\begin{proof}
Let $\mu_n \colon \mathsf{LC}(Q) \to {\N}^{< \N}$ and $\Psi_n \colon \mathsf{LC}(Q) \to \mathcal{Z}_P$, $n \in \N$, be the Borel measurable mappings given by Lemma \ref{BorelSequenceOfSubChains}. For every $\Gamma \in \mathcal{Z}$, let $l(\Gamma) \in \N$ be the number satisfying $\Gamma \in \big( \mathsf{C} (Q) \big)^{l(\Gamma)}$. Let us define a function $\lambda \colon \mathsf{LC}(Q) \times \N \to \N$ by $\lambda (X,n) = l \big( \Psi_n (X) \big)$. Clearly, $\lambda$ is Borel measurable. For every $n \in \N \setminus \{ 1 \}$, define a mapping $\nu_n \colon \mathsf{LC}(Q) \to {\N}^{< \N}$ such that for every $X \in \mathsf{LC}(Q)$, roughly speaking, the $i$-th element of $\nu_n (X)$ is the length of the subsequence of $\Psi_n (X)$ corresponding to the $i$-th member of $\Psi_{n-1} (X)$. Formally, if $\mu_n (X) = (j_1, \dotsc, j_m)$, we define $\nu_n (X)$ as follows: If $m=1$, let $\nu_n (X) = \big( \lambda (X,n) \big)$. If $m=2$, define
\[ \nu_n (X) = \big( j_2-1, \ \lambda (X,n)+1-j_2 \big) = \big( j_2-j_1, \ \lambda (X,n)+1-j_m \big).\]
Finally, if $m \geq 3$, let
\[ \nu_n (X) = \big( j_2-j_1, \ j_3-j_2, \, \dotsc \, , j_m-j_{m-1}, \ \lambda (X,n)+1-j_m \big). \]
For any $n \in \N \setminus \{ 1 \}$, since $\mu_n$ is Borel measurable, so is $\nu_n$. Note that for every $n \in \N \setminus \{ 1 \}$ and $X \in \mathsf{LC}(Q)$, if $\nu_n (X) = (i_1, \dotsc, i_k)$ and $\mu_n (X) = (j_1, \dotsc, j_m)$, then $k=m=\lambda (X,n-1)$ and $i_1+\dots+i_k=\lambda (X,n)$.

Let $W:= \big\lbrace (a,b,k) \in I \times I \times \N \, ; \ a<b \big\rbrace$, $\displaystyle T:= \bigoplus_{p \in \N} I^p$ and define a mapping $\phi \colon W \to T$ by
\[ \phi (a,b,k) = \bigg( a+\frac{(i-1)(b-a)}{k} \bigg)_{i=1}^{k+1}. \]
In other words, for any $(a,b,k) \in W$, if $\phi (a,b,k) = (t_0, \dotsc, t_m)$, then $m=k$, $a=t_0 < t_1 < \dots < t_m=b$ and $t_i - t_{i-1} = \frac{b-a}{k}$ for $i=1, \dotsc, m$. Clearly, $\phi$ is continuous. In particular, $\phi$ is Borel measurable and so is the mapping $\tau_1 \colon \mathsf{LC}(Q) \to T$ given by
$\tau_1 (X) = \phi \big( 0,1,\lambda(X,1) \big)$. Define a mapping $\tau_2 \colon \mathsf{LC}(Q) \to T$ as follows: Let $X \in \mathsf{LC}(Q)$ be given and denote $m:=\lambda(X,1)$. If $\tau_1 (X) =(t_0, \dotsc, t_m)$, $\nu_2 (X)=(k_1, \dotsc, k_m)$ and $\phi (t_{i-1}, t_i, k_i)= \big(s_0^i, s_1^i, \dotsc, s_{k_i}^i \big)$ for every $i=1, \dotsc, m$, then we define
\[ \tau_2 (X) = \big( 0, s_1^1, \dotsc, s_{k_1}^1, s_1^2, \dotsc, s_{k_2}^2, \dotsc, s_1^m, \dotsc, s_{k_m}^m \big). \]
Note that $s_{k_m}^m=t_m=1$ and $s_0^i = t_{i-1} = s_{k_{i-1}}^{i-1}$ for $i=2, \dotsc, m$. Moreover, if we relabel the elements of $\tau_2 (X)$ so that $\tau_2 (X) = (s_0, \dotsc, s_l)$, then $l=k_1+ \dots +k_m = \lambda(X,2)$.

If we keep repeating the process used to construct $\tau_2$, we obtain Borel measurable mappings $\tau_n \colon \mathsf{LC}(Q) \to T$, $n \in \N$, such that for every $X \in \mathsf{LC}(Q)$ and $n \in \N$, if $\tau_n (X) =(t_0, \dotsc, t_m)$, then $m=\lambda(X,n)$ and
\[ \tau_{n+1} (X) = \big( 0, s_1^1, \dotsc, s_{k_1}^1, s_1^2, \dotsc, s_{k_2}^2, \dotsc, s_1^m, \dotsc, s_{k_m}^m \big), \]
where $(k_1, \dotsc, k_m)=\nu_{n+1} (X)$ and $\big(s_0^i, s_1^i, \dotsc, s_{k_i}^i \big)=\phi (t_{i-1}, t_i, k_i)$ for every $i=1, \dotsc, m$.

For every $n \in \N$, define a mapping $\Phi_n \colon \mathsf{LC}(Q) \to \mathcal{K}(I \times Q)$ as follows: For any $X \in \mathsf{LC}(Q)$, if $\Psi_n (X)=(K_1, \dotsc, K_m)$ and $\tau_n (X) = (t_0, \dotsc, t_m)$, let
\[ \Phi_n (X) = \bigcup_{i=1}^m \big( [t_{i-1},t_i] \times K_i \big). \]
It is easy to see that $\Phi_n$ is Borel measurable for each $n \in \N$. For every $X \in \mathsf{LC}(Q)$ and $n \in \N$, define a mapping $\psi_n^X \colon I \to \mathcal{F}(X)$ by
\[ \psi_n^X (t) = \big( \Phi_n (X) \big)_t = \big\lbrace x \in Q; \ (t,x) \in \Phi_n (X) \big\rbrace = \big\lbrace x \in X; \ (t,x) \in \Phi_n (X) \big\rbrace. \]
Clearly, if $\Psi_n (X)=(K_1, \dotsc, K_m)$ and $\tau_n (X) = (t_0, \dotsc, t_m)$, then
\begin{equation*}
    \psi_n^X (t) = \begin{cases}
    K_1 & \textup{if } t=0 \\
    K_m & \textup{if } t=1 \\
    K_i & \textup{if } t_{i-1}<t<t_i, \ i = 1, \dotsc, m \\
    K_i \cup K_{i+1} & \textup{if } t=t_i, \ i = 1, \dotsc, m-1.
\end{cases} \ \ 
\end{equation*}
For every $X \in \mathsf{LC}(Q)$, it is fairly straightforward to verify that the mappings $\psi_1^X, \psi_2^X, \psi_3^X, \dotsc$ satisfy the assumptions of Theorem \ref{GeneralMapping}. Thus, the mapping $\Phi \colon \mathsf{LC}(Q) \to C(I,Q)$ given by
\[ \big\lbrace \Phi(X)(t) \big\rbrace = \bigcap_{n \in \N} \psi_n^X (t) \]
is well-defined and for every $X \in \mathsf{LC}(Q)$, the image of $I$ under $\Phi(X)$ is equal to $X$. It remains to show that $\Phi$ is Borel measurable. By Lemma \ref{SpaceOfCtsFctionsEmbeddedIntoHyperspace}, it will be enough if we prove that the mapping $\widehat{\Phi} \colon \mathsf{LC}(Q) \to \mathcal{K}(I \times Q)$ given by
\[ \widehat{\Phi} (X) = \mathrm{graph} \big( \Phi (X) \big) \]
is Borel measurable. However, we clearly have
\[ \widehat{\Phi} (X) = \bigcap_{n \in \N} \Phi_n (X) \]
for every $X \in \mathsf{LC}(Q)$. Thus, by Lemma \ref{IntersectionOfBorelMapsToHyperspaceIsBorel}, $\widehat{\Phi}$ is Borel measurable.
\end{proof}

\begin{corollary}\label{CorMainTheorem}
There is a Borel measurable bijection $\phi:C(I, Q)\to \mathsf{LC}(Q)$ such that $\phi(f)$ is homeomorphic to the image of $f$ for every $f \in C(I, Q)$.
\end{corollary}

\begin{proof}
Similarly to \cite[326-327]{Gao}.
Let $\Phi \colon \mathsf{LC}(Q) \to C(I,Q)$ be the mapping provided by Theorem \ref{MainTheorem}. Clearly, $\Phi$ is a Borel measurable injection. On the other hand, it is possible to construct a Borel measurable injection $\chi: C(I, Q) \to \mathsf{LC}(Q)$ such that $\chi(f)$ homeomorphic to $f(I)$ for each $f \in C(I, Q)$. This can be done as follows: Fix a Borel measurable injection $\theta : C(I, Q)\to [0,1]$. Then let $\chi(f) = \{ \theta (f) \} \times f(I)$ for every $f \in C(I, Q)$.

A standard Cantor Bernstein argument applied to $\Phi$ and $\chi$ gives us the desired Borel measurable bijection $\phi:C(I, Q)\to \mathsf{LC}(Q)$.
\end{proof}

\section{Peano continua are arcwise connected in a Borel way}

It is well-known that Peano continua are arcwise connected. In this section we prove that an arc connecting two points in a Peano continuum can be chosen in a Borel measurable way.

\begin{lemma}\label{BorelChoiceOfMinimumPoint}
Let $Y$ be a compact metrizable space and $\varphi \colon Y \to \R$ a continuous function. There is a Borel measurable mapping $\nu \colon \mathcal{K}(Y) \to Y$ such that $\nu (K) \in K$ and $\displaystyle \varphi \big( \nu (K) \big) = \min \big\lbrace \varphi(y) \, ; \, y \in K \big\rbrace$ for every $K \in \mathcal{K}(Y)$.
\end{lemma}
\begin{proof}
Define $\mathcal{D}:= \big\lbrace (K,t) \in \mathcal{K}(Y) \times \R \, ; \ \exists \, y \in K : \varphi(y) \leq t \big\rbrace$ and define a mapping $\xi \colon \mathcal{D} \to \mathcal{K}(Y)$ by $\xi (K,t) = \big\lbrace y \in K \, ; \ \varphi (y) \leq t \big\rbrace$. Fix a number $s \in \R$ such that $\varphi (y) > s$ for every $y \in Y$. Then we have $\xi (K,t) = K \cap \varphi^{-1} \big( [s , t] \big)$ for all $(K,t) \in \mathcal{D}$. By Lemmata \ref{intersectionIsBorel} and \ref{CompactMapstoPreimageIsBorel}, it easily follows that $\xi$ is Borel measurable. The function $\mu \colon \mathcal{K}(Y) \to \R$ given by $\mu(K) = \min \big\lbrace \varphi(y) ; \, y \in K \big\rbrace$ is continuous (this is a very easy exercise). Let $\sigma \colon \mathcal{K}(Y) \to Y$ be a Borel measurable mapping satisfying $\sigma (L) \in L$ for every $L \in \mathcal{K}(Y)$. Then we can define the desired mapping by $\nu (K) = \sigma \big( \xi \big( K, \mu (K) \big) \big)$.
\end{proof}

By \cite[4.33]{Nadler}, for every compact metrizable space $Z$, there exists a continuous function $\varphi \colon \mathcal{K}(Z) \to \R$ such that $\varphi \big( \{ z \} \big)=0$ for each $z \in Z$ and $\varphi(K) < \varphi(L)$ whenever $K,L \in \mathcal{K}(Z)$ satisfy $K \subsetneqq L$. Such a function $\varphi$ is called a Whitney map.

\begin{lemma}\label{BorelChoiceOfMinimalElementInclusion}
There exists a Borel measurable mapping $\Upsilon \colon \mathcal{K} \big( \mathcal{K} (I) \big) \to \mathcal{K} (I)$ such that for every $\mathcal{L} \in \mathcal{K} \big( \mathcal{K} (I) \big)$, the set $\Upsilon (\mathcal{L})$ is a minimal element of the family $\mathcal{L}$ with respect to set inclusion.
\end{lemma}
\begin{proof}
Let $\varphi \colon \mathcal{K}(I) \to \R$ be a Whitney map. By Lemma \ref{BorelChoiceOfMinimumPoint}, there is a Borel measurable mapping $\nu \colon \mathcal{K} \big( \mathcal{K} (I) \big) \to \mathcal{K} (I)$ satisfying $\nu (\mathcal{L}) \in \mathcal{L}$ and $\varphi \big( \nu (\mathcal{L}) \big) = \min \big\lbrace \varphi(K) \, ; \, K \in \mathcal{L} \big\rbrace$ for all $\mathcal{L} \in \mathcal{K} \big( \mathcal{K} (I) \big)$. Since $\varphi$ is a Whitney map, it is easy to see that $\nu$ is the desired mapping.
\end{proof}

\begin{lemma}\label{AssigningTheFamilyOfAdmissibleCompactaIsBorel1}
For every continuous mapping $f \colon I \to Q$, the set
\[ \Lambda(f) := \big\lbrace K \in \mathcal{K}(I) \, ; \ f(\min K) = f(0) \, , \ f(\max K)=f(1) \big\rbrace \]
is nonempty and compact. Moreover, the mapping $\Lambda \colon C(I,Q) \to \mathcal{K}\big( \mathcal{K}(I) \big)$ defined by $f \mapsto \Lambda (f)$ is Borel measurable.
\end{lemma}
\begin{proof}
For every $f \in C(I,Q)$, the set $\Lambda (f)$ is nonempty since it obviously contains $I$. Let
\[ \mathcal{F}:= \big\lbrace (f,K) \in C(I,Q) \times \mathcal{K}(I) \, ; \ f(\min K) = f(0) \, , \ f(\max K)=f(1) \big\rbrace. \]
\begin{claim}\label{claimmm}
The set $\mathcal{F}$ is closed in $C(I,Q) \times \mathcal{K}(I)$.
\end{claim}
\begin{claimproof}
Define functions $M_1 \colon \mathcal{K}(I) \to I$, $M_2 \colon \mathcal{K}(I) \to I$ by $M_1 (K) = \min K$ and $M_2 (K) = \max K$. Define mappings $\Psi_1, \Psi_2, \Psi_3, \Psi_4 \colon C(I,Q) \times \mathcal{K}(I) \to Q$ by $\Psi_1 (f,K) = f(M_1(K))$, $\Psi_2 (f,K) = f(0)$, $\Psi_3 (f,K) = f(M_2(K))$ and $\Psi_4 (f,K) = f(1)$. Trivially, $\Psi_2$ and $\Psi_4$ are continuous. Moreover, since $M_1$ and $M_2$ are continuous, it is easy to see that so are $\Psi_1$ and $\Psi_3$. Therefore, the set
\[ \big\lbrace (f,K) \in C(I,Q) \times \mathcal{K}(I) \, ; \ \Psi_1 (f,K)=\Psi_2 (f,K) \, , \ \Psi_3 (f,K)=\Psi_4 (f,K) \big\rbrace \]
is closed. Clearly, this set is equal to $\mathcal{F}$. \claimend
\end{claimproof}
Since $\mathcal{F}$ is closed and $\mathcal{K}(I)$ is compact, the vertical sections of $\mathcal{F}$ are compact. In other words, $\Lambda (f)$ is compact for every $f \in C(I,Q)$. Finally, $\Lambda$ is Borel measurable by Lemma \ref{ClosedSetInProductOfPolishAndCompact}.
\end{proof}

For every $K \in \mathcal{K}(I)$, let
\begin{align*}
    \alpha (K):&= \big\lbrace (s,t) \in I \times I \, ; \ s<t \textup{ and } [s,t] \cap K = \{ s,t \} \big\rbrace\\
    &= \big\lbrace (s,t) \in K \times K \, ; \ s<t \textup{ and } \forall \, u \in I: s<u<t \implies u \notin K \big\rbrace.
\end{align*}

\begin{lemma}\label{AssigningTheFamilyOfAdmissibleCompactaIsBorel2}
For every continuous mapping $f \colon I \to Q$, the set
\[ \Gamma (f):= \big\lbrace K \in \mathcal{K}(I) \, ; \ \forall \, (s,t) \in \alpha(K) : f(s)=f(t) \big\rbrace \]
is nonempty and compact. Moreover, the mapping $\Gamma \colon C(I,Q) \to \mathcal{K} \big( \mathcal{K}(I) \big)$ defined by $f \mapsto \Gamma (f)$ is Borel measurable.
\end{lemma}
\begin{proof}
Clearly, $\alpha(I)=\emptyset$. Thus, for every $f \in C(I,Q)$, the set $\Gamma (f)$ is nonempty as it contains $I$. Let $\mathcal{F}:= \big\lbrace (f,K) \in C(I,Q) \times \mathcal{K}(I) \, ; \ \forall \, (s,t) \in \alpha(K) : f(s)=f(t) \big\rbrace$.
\begin{claim}
The set $\mathcal{F}$ is closed in $C(I,Q) \times \mathcal{K}(I)$.
\end{claim}
\begin{claimproof}
Let us prove that the set $\mathcal{G}:= \big( C(I,Q) \times \mathcal{K}(I) \big) \setminus \mathcal{F}$ is open. Given $(f,K) \in \mathcal{G}$, fix a pair $(s,t) \in \alpha (K)$ with $f(s) \neq f(t)$. Let $U$ and $V$ be disjoint open subsets of $Q$ satisfying $f(s) \in U$ and $f(t) \in V$. By the continuity of $f$, there is $\delta > 0$ such that $f(u) \in U$ for every $u \in I$ with $|u-s| \leq \delta$ and $f(v) \in V$ for every $v \in I$ with $|v-t| \leq \delta$. Since $s<t$, we can assume that $s+\delta < t-\delta$. Let $J_1 := I \cap [s- \delta , s+ \delta]$ and $J_2 := I \cap [t- \delta , t+ \delta]$. Then $\mathcal{U} := \big\lbrace g \in C(I,Q) \, ; \ g(J_1) \subseteq U , \ g(J_2) \subseteq V \big\rbrace$ contains $f$ and it is clear (as $J_1$, $J_2$ are compact) that $\mathcal{U}$ is open in $C(I,Q)$. Furthermore, the set
\begin{align*}
    \mathcal{V}:=\big\lbrace L \in \mathcal{K}(I) \, ; \ &L \subseteq [0,s+\delta) \cup (t-\delta, 1] ,\\
    &L \cap (s-\delta, s+\delta) \neq \emptyset, \ L \cap (t-\delta, t+\delta) \neq \emptyset\big\rbrace
\end{align*}
is open in $\mathcal{K}(I)$ and it contains $K$. Hence, $\mathcal{U} \times \mathcal{V}$ is a neighbourhood of $(f,K)$. It remains to verify that $\mathcal{U} \times \mathcal{V} \subseteq \mathcal{G}$. Let $(g,L) \in \mathcal{U} \times \mathcal{V}$ be given and let $u:=\max \big( L \cap [0,s+\delta] \big)$, $v:=\min \big( L \cap [t-\delta,1] \big)$. It is easy to see that $(u,v) \in \alpha(L)$. Also, since $u \in J_1$ and $v \in J_2$, we have $g(u) \in U$ and $g(v) \in V$. In particular, $g(u) \neq g(v)$, which proves that $(g,L) \in \mathcal{G}$. \claimend
\end{claimproof}
Since $\mathcal{F}$ is closed and $\mathcal{K}(I)$ is compact, the vertical sections of $\mathcal{F}$ are compact. In other words, $\Gamma (f)$ is compact for every $f \in C(I,Q)$. Finally, $\Gamma$ is Borel measurable by Lemma \ref{ClosedSetInProductOfPolishAndCompact}.
\end{proof}

\begin{lemma}\label{TransformationOfCtsMappings}
Let $\mathcal{B}:=\big\lbrace (X,x,y) \in \mathsf{LC}(Q) \times Q \times Q \, ; \ x,y \in X \big\rbrace$. There is a Borel measurable mapping $T \colon \mathcal{B} \to C(I,Q)$ such that for every $(X,x,y) \in \mathcal{B}$, if $T (X,x,y) = h$, then $h(0)=x$, $h(1)=y$ and $h(I)=X$.
\end{lemma}
\begin{proof}
For every $f \in C(I,Q)$ and every two points $a,b \in I$, define a mapping $f_{a,b} \colon I \to Q$ by
\begin{equation*}
    f_{a,b} (t) = \begin{cases}
    f(a-3at) & \textup{if } 0 \leq t < \frac{1}{3} \\
    f(3t-1) & \textup{if } \frac{1}{3} \leq t \leq \frac{2}{3} \\
    f(3bt-3t-2b+3) & \textup{if } \frac{2}{3} < t \leq 1 .
\end{cases}
\end{equation*}
It is easy to see that $f_{a,b}$ is well-defined and continuous. Moreover, we have $f_{a,b} (0)=f(a)$, $f_{a,b} (1)=f(b)$ and $f_{a,b} (I) = f(I)$. Let us define a mapping $\phi \colon C(I,Q) \times I \times I \to C(I,Q)$ by $\phi (f,a,b)=f_{a,b}$. It is straightforward to prove that $\phi$ is continuous. Let $\mathcal{A}:= \big\lbrace (f,x) \in C(I,Q) \times Q \, ; \ x \in f(I) \big\rbrace$ and define a mapping $\gamma \colon \mathcal{A} \to \mathcal{K}(I)$ by $\gamma (f,x) = f^{-1} \big( \{ x \} \big)$. By Lemma \ref{CompactMapstoPreimageIsBorel}, $\gamma$ is Borel measurable. Let $\sigma \colon \mathcal{K}(I) \to I$ be a Borel measurable mapping such that $\sigma(K) \in K$ for every $K \in \mathcal{K}(I)$. Then $\tau := \sigma \circ \gamma$ is a Borel measurable mapping such that $f \big( \tau (f,x) \big) = x$ for every $(f,x) \in \mathcal{A}$. Let $\Phi$ be the mapping from Theorem \ref{MainTheorem}. Finally, define the desired mapping $T \colon \mathcal{B} \to C(I,Q)$ by
\[ T(X,x,y) = \phi \Big( \Phi(X), \tau\big( \Phi(X),x \big), \tau\big( \Phi(X),y \big) \Big). \]
It is easy to verify that all of the requirements we have on $T$ are met.
\end{proof}

Recall that a mapping $f \colon X \to Y$, where $X$ and $Y$ are topological spaces, is said to be monotone if $f$ is continuous and $f^{-1} \big( \{ y \} \big)$ is connected for every $y \in f(X)$. The following lemma is an immediate consequence of \cite[8.22]{Nadler}.

\begin{lemma}\label{MonotoneImageOfIntervalIsArc}
If $g \colon I \to Q$ is a monotone mapping such that $g(0)\neq g(1)$, then $g(I)$ is an arc.
\end{lemma}

\begin{theorem}\label{ArcsInPeanoContinua}
Let $\mathcal{D}:= \big\lbrace (X,x,y) \in \mathsf{LC}(Q) \times Q \times Q \, ; \ x,y \in X, \, x \neq y \big\rbrace$. There exists a Borel measurable mapping $A \colon \mathcal{D} \to \mathcal{K}(Q)$ such that for every $(X,x,y) \in \mathcal{D}$, the set $A(X,x,y)$ is an arc contained in $X$ whose endpoints are $x$ and $y$.
\end{theorem}
\begin{proof}
Let $\mathcal{B}:=\big\lbrace (X,x,y) \in \mathsf{LC}(Q) \times Q \times Q \, ; \ x,y \in X \big\rbrace$ and consider the mapping $T \colon \mathcal{B} \to C(I,Q)$ from Lemma \ref{TransformationOfCtsMappings}. Let $\Lambda \colon C(I,Q) \to \mathcal{K}\big( \mathcal{K}(I) \big)$ and $\Gamma \colon C(I,Q) \to \mathcal{K}\big( \mathcal{K}(I) \big)$ be the mappings from Lemma \ref{AssigningTheFamilyOfAdmissibleCompactaIsBorel1} and Lemma \ref{AssigningTheFamilyOfAdmissibleCompactaIsBorel2} respectively. Clearly, for every $f \in C(I,Q)$, we have $I \in \Lambda(f) \cap \Gamma (f)$ and, therefore, $\Lambda(f) \cap \Gamma (f) \neq \emptyset$. Hence, we can define a mapping $\Omega \colon C(I,Q) \to \mathcal{K}\big( \mathcal{K}(I) \big)$ by $\Omega(f)=\Lambda(f) \cap \Gamma (f)$. By Lemma \ref{IntersectionOfBorelMapsToHyperspaceIsBorel}, $\Omega$ is Borel measurable. Let $\Upsilon \colon \mathcal{K} \big( \mathcal{K} (I) \big) \to \mathcal{K} (I)$ be the mapping provided by Lemma \ref{BorelChoiceOfMinimalElementInclusion}. Finally, by Lemma \ref{CtsFunctionAndCompactSetMapstoImageIsCts}, the mapping $\Delta \colon C(I,Q) \times \mathcal{K}(I) \to \mathcal{K}(Q)$ given by $\Delta (f,K)=f(K)$ is continuous. Define the desired mapping $A \colon \mathcal{D} \to \mathcal{K}(Q)$ by
\[ A(X,x,y) = \Delta \big( T(X,x,y), \, (\Upsilon \circ \Omega \circ T)(X,x,y) \big). \]
Clearly, $A$ is Borel measurable. It remains to show that $A(X,x,y)$ is an arc in $X$ with endpoints $x,y$, whenever $(X,x,y) \in \mathcal{D}$. At this point, we could simply refer the reader to the corresponding part of the proof of \cite[8.23]{Nadler}, but, for the sake of completeness, let us present the proof here. Let $(X,x,y) \in \mathcal{D}$ be given and denote $f:=T(X,x,y)$, $\mathcal{L}:=\Omega (f)$, $K:=\Upsilon (\mathcal{L})$. Then $f(0)=x$, $f(1)=y$, $f(I)=X$,
\[ \mathcal{L}= \big\lbrace  L \in \mathcal{K}(I) \, ; \ \forall \, (s,t) \in \alpha(L) : f(s)=f(t) \, , \ f(\min L)=f(0) \, , \ f(\max L)=f(1) \big\rbrace , \]
$K \in \mathcal{L}$ is a minimal element of $\mathcal{L}$ with respect to set inclusion and $A(X,x,y)=\Delta(f,K)=f(K)$.
\begin{claim}\label{Claim1InTheProofOfTheArcTheorem}
For any two points $a,b \in K$ with $a \leq b$ and $f(a)=f(b)$, we have $K \cap [a,b] = \{ a,b \}$.
\end{claim}
\begin{claimproof}
Let $a,b \in K$ be arbitrary and assume that $a<b$ and $f(a)=f(b)$. The set $K_0 := K \setminus (a,b)$ satisfies $K_0 \cap [a,b] = \{ a,b \}$. Hence, it suffices to show that $K_0 = K$. Clearly, $K_0$ is a nonempty compact subset of $I$. Also, $\min K_0 = \min K$, $\max K_0 = \max K$ and, therefore (since $K \in \mathcal{L}$), we have $f(\min K_0)=f(0)$, $f(\max K_0)=f(1)$. Now, given $(s,t) \in \alpha(K_0)$, we would like to verify that $f(s)=f(t)$. This equality is obvious if $(s,t)=(a,b)$. Hence, assume that $(s,t) \neq (a,b)$. Then, as $K_0 \cap [a,b] = \{ a,b \}$, we have either $s<t \leq a$, or $b \leq s < t$. In both cases it immediately follows from the definition of $K_0$ that $(s,t) \in \alpha(K)$ and, thus, $f(s)=f(t)$. We have just shown that $K_0 \in \mathcal{L}$. Therefore, since $K_0 \subseteq K$ and $K$ is a minimal element of $\mathcal{L}$ with respect to inclusion, it follows that $K_0=K$. \claimend
\end{claimproof}
Note that for every $u \in [\min K, \max K] \setminus K$, there is exactly one pair $(s_u,t_u) \in \alpha(K)$ satisfying $s_u<u<t_u$. Define a mapping $g \colon I \to Q$ by
\begin{equation*}
    g(u) = \begin{cases}
    f(u) & \textup{if } u \in K \\
    f(0) & \textup{if } 0 \leq u < \min K \\
    f(1) & \textup{if } \max K < u \leq 1 \\
    f(s_u) & \textup{if } u \in [\min K, \max K] \setminus K.
\end{cases}
\end{equation*}
It is fairly easy to prove that $g$ is continuous. Also, since it is clear that $g(I)=g(K)=f(K)=A(X,x,y)$ and $f(K) \subseteq f(I) = X$, it suffices to show that $g(I)$ is an arc with endpoints $x,y$.
\begin{claim}\label{Claim2InTheProofOfTheArcTheorem}
The mapping $g$ is monotone.
\end{claim}
\begin{claimproof}
Since $g$ is continuous, it suffices to show that $g^{-1} \big( \{ z \} \big)$ is connected for every $z \in g(I)$. Let $z \in g(I)$ be arbitrary and let $M:= K \cap g^{-1} \big( \{ z \} \big)$. As $g(I)=g(K)$ and $g$ is continuous, $M$ is nonempty and compact. Define $a:=\min M$, $b := \max M$. Clearly, $f(a)=g(a)=z=g(b)=f(b)$. Thus, by Claim \ref{Claim1InTheProofOfTheArcTheorem}, we have $K \cap [a,b] = \{ a,b \}$. It follows from the definition of $g$ that $[\min K, \max K] \cap g^{-1} \big( \{ z \} \big) = [a,b]$. Now it is easy to see that $g^{-1} \big( \{ z \} \big)$ is equal to one of the intervals $[a,b]$, $[0,b]$, $[a,1]$. \claimend
\end{claimproof}
Since $g$ is monotone and $g(0)=f(0)=x \neq y=f(1)=g(1)$, it follows from Lemma \ref{MonotoneImageOfIntervalIsArc} that $g(I)$ is an arc. Finally, to show that $x$ and $y$ are the endpoints of the arc $g(I)$, let $a:= \max g^{-1} \big( \{ x \} \big)$ and $b:= \min g^{-1} \big( \{ y \} \big)$. Then $g^{-1} \big( \{ x \} \big) = [0,a]$, $g^{-1} \big( \{ y \} \big) = [b,1]$ and $a<b$, which shows that $g \big( (a,b) \big) = g(I) \setminus \{ x,y \}$. Hence, $g(I) \setminus \{ x,y \}$ is a connected set. This clearly implies that $x$ and $y$ are the endpoints of the arc $g(I)$.
\end{proof}

\bibliographystyle{alpha}
\bibliography{citace}
\end{document}